\documentclass[12pt,reqno]{amsart}
\usepackage{geometry}
\usepackage[T1]{fontenc}
\usepackage[latin1]{inputenc}
\usepackage[italian,english]{babel}
\usepackage{amsmath, amsfonts, amsthm}
\usepackage{paralist}
\usepackage{verbatim}
\numberwithin{equation}{section}
\newtheorem{thm}{Theorem}[section]
\newtheorem{lem}[thm]{Lemma}
\newtheorem{cor}[thm]{Corollary}
\newtheorem{prop}[thm]{Proposition}
\newtheorem{exa}{Example}

\newtheorem{oppr}{Open Problem}
\theoremstyle{definition}
\newtheorem{rem}[thm]{Remark}
\theoremstyle{remark}

\newcommand{\ds}{\displaystyle}

\date{}

\title{Comparison results for solutions to the anisotropic Laplacian with Robin boundary conditions}
\author[R. Sannipoli]{Rossano Sannipoli}
\address{Dipartimento di Matematica e Applicazioni ``R. Caccioppoli'', Universit\`a degli studi di Napoli Federico II \\ Via Cintia, Complesso Universitario Monte S. Angelo, 80126 Napoli, Italy.}
\email{rossano.sannipoli@unina.it}

\begin{document}
\maketitle 
\markright{COMPARISON RESULTS FOR SOLUTIONS TO THE ANISOTROPIC LAPLACIAN}
\begin{abstract}
	
	In this paper we consider PDE's problems involving the anisotropic Laplacian operator, with Robin boundary conditions. By means of Talenti techniques, widely used in the last decades, we prove a comparison result between the solutions of the above-mentioned problems and the solutions of the symmetrized ones. As a consequence of these results, a Bossel-Daners type inequality can be shown in dimension 2.
\end{abstract}
\section{Introduction}
Let $\Omega\subset \mathbb{R}^n$ be an open bounded set, with Lipschitz boundary. Let us consider the following anisotropic problem with Robin boundary conditions
\begin{equation} \label{ATPR}
\begin{cases}
-\mathrm{div} (H(\nabla u)H_{\xi}(\nabla u))=f & \mbox{in}\ \Omega\vspace{0.2cm}\\
H(\nabla u)H_{\xi}(\nabla u)\cdot \nu +\beta H(\nu) u=0&\mbox{on}\ \partial \Omega\vspace{0.2cm},
\end{cases}
\end{equation}
where $f\ge0$ (not identically zero) belongs to $L^2(\Omega)$, $H$ is a sufficiently smooth norm in $\mathbb{R}^n$, $\nu$ is the outer unit normal to $\partial\Omega$ and $\beta > 0$ is a positive real parameter.\\
A weak solution to problem \eqref{ATPR} is a function $u\in H^1(\Omega)$ that satisfies
\begin{equation} \label{ELE}
\int_{\Omega} H(\nabla u)H_{\xi}(\nabla u) \cdot \nabla \varphi \,dx + \beta\int_{\partial\Omega} H(\nu)u\varphi \,d\mathcal{H}^{n-1} = \int_{\Omega} f \varphi \,\,\,\,\,\,\,\,\,\,\,\,\,\, \forall \varphi \in H^1(\Omega).
\end{equation}
The aim of the paper is to estabilish a comparison result with the solution to the following symmetrized problem
\begin{equation} \label{SATPR}
\begin{cases}
-\mathrm{div} (H(\nabla v)H_{\xi}(\nabla v))=f^{\star} & \mbox{in}\ \Omega^{\star}\vspace{0.2cm}\\
H(\nabla v)H_{\xi}(\nabla v)\cdot \nu +\beta H(\nu) v=0&\mbox{on}\ \partial \Omega^{\star}\vspace{0.2cm},
\end{cases}
\end{equation}
where $f^{\star}$ is the convex symmetrization of $f$ and $\Omega^{\star}$ is a set homothetic to the Wulff Shape $\mathcal{W}$ such that $|\Omega^{\star}|=|\Omega|$ (for the exact definitions, see section 2).\\ \\
It is well known that, when $H$ is the euclidean norm in $\mathbb{R}^n$, then the convex symmetrization coincides with the Schwarz symmetrization (to know more about Schwarz symmetrization see \cite{K}). Moreover if the Robin boundary condition is replaced by the Dirichlet boundary condition, problem \eqref{ATPR} becomes
\begin{equation} \label{TAL}
	\begin{cases}
	-\Delta u = f & \mbox{in} \,\,\, \Omega \\
	u = 0 & \mbox{on} \,\,\, \partial \Omega,
	\end{cases}
\end{equation}
with $f\in L^2(\Omega)$ (non-negative and not identically zero) and $\Omega$ an open subset of $\mathbb{R}^n$. This problem has been widely studied, with Talenti being the pioneer in this direction. Indeed the author (see \cite{T}) proved, via rearrangements arguments, that the Schwarz symmetrization of the solution to problem \eqref{TAL} is pointwise bounded  by the solution to the following symmetrized problem
\begin{equation}
	\begin{cases}
	-\Delta v = f^{\sharp} & \mbox{in} \,\,\, \Omega^{\sharp} \\
	v = 0 & \mbox{on} \,\,\, \partial \Omega^{\sharp},
	\end{cases}
\end{equation}
with $f^{\sharp}$ being the Schwarz decreasing rearrangement of $f$ and $\Omega^{\sharp}$ the ball centered at the origin having the same measure as $\Omega$. 
Talenti, with his techniques, gave birth to a series of generalizations and results that still now take his name. For instance see \cite{AFLT, ALT, T1, T2} for generalizations to other kind of operators.\\
When we have Robin boundary conditions with positive parameter, problem \eqref{TAL} becomes
\begin{equation} \label{TALROB}
\begin{cases}
-\Delta u = f & \mbox{in} \,\,\, \Omega \\
\ds \frac{\partial u}{\partial \nu} + \beta u = 0 & \mbox{on} \,\,\, \partial \Omega.
\end{cases}
\end{equation}
To our knowledge, in literature, there are few comparison results \'a la Talenti for this kind of problem. A result of this type has been proved only recently by \cite{ANT}, where they have highlighted the importance of the dependence on the dimension of the space. The authors, infact, managed to compare the Lorentz norm (see \cite{L}) of the solution to problem \eqref{TALROB} with that of the solution to the symmetrized problem
\begin{equation} \label{TALROBS}
	\begin{cases}
	-\Delta v = f^{\sharp} & \mbox{in} \,\,\, \Omega^{\sharp} \\
	\ds \frac{\partial v}{\partial \nu} + \beta u = 0 & \mbox{on} \,\,\, \partial \Omega^{\sharp},
	\end{cases}
\end{equation}
where the exponents of these norms depend on the dimension of the space. In particular they proved, for $n\ge 2$, that 
\begin{equation*}
\|u\|_{L^{p,1}(\Omega)} \le \|v\|_{L^{p,1}(\Omega^{\sharp})} \,\,\,\,\,\,\, \text{for all} \,\,\, 0<p \le \frac{n}{2n-2}
\end{equation*} 
and
\begin{equation*}
\|u\|_{L^{2p,2}(\Omega)} \le \|v\|_{L^{2p,2}(\Omega^{\sharp})}
\,\,\,\,\,\,\, \text{for all} \,\,\, 0<p \le \frac{n}{3n-4},
\end{equation*}
with $u$ solution to \eqref{TALROB} and $v$ to \eqref{TALROBS}. Moreover, when $f \equiv 1$ in $\Omega$ and n=2, they showed that
\begin{equation} 
u^{\sharp}(x) \le v(x) \;\;\;\;\; x\in \Omega^{\sharp},
\end{equation}
and, for $n\ge 3$, that 
 \begin{equation*}
 \begin{split}
&\|u\|_{L^{p,1}(\Omega)}\le \|v\|_{L^{p,1}(\Omega^{\sharp})} \\
&\|u\|_{L^{2p,2}(\Omega)}\le \|v\|_{L^{2p,2}(\Omega^{\sharp})},
\end{split} 
\end{equation*}
for all $0<p\le \frac{n}{n-2}$.\\
With a different approach \cite{BG} proved that 
\begin{equation*}
	\|u\|_{L^1(\Omega)} \le \|v\|_{L^1(\Omega^{\sharp})}.
\end{equation*}
In \cite{S} one can find the computation of the shape derivative of the $L^{\infty}$ and $L^p$ norms of the solution to the problem \eqref{TALROB} and the proof of the stationarity of the ball, when fixing the volume.\\
The main theorems of this paper show that the proves proposed by \cite{ANT} can be adapted even when considering the anisotropic Laplacian operator. Let us now state these results.

\begin{thm} \label{FNM2}
	Let be $n\ge 2$. If $u$ and $v$ are the solutions to problems \eqref{ATPR} and \eqref{SATPR} respectivley, then 
	\begin{equation} \label{P1f}
	\|u\|_{L^{p,1}(\Omega)} \le \|v\|_{L^{p,1}(\Omega^{\star})} \;\;\;\;\;\;\; \text{for all} \;\;\;\;\; 0 < p \le \frac{n}{2n-2}
	\end{equation}
	and
	\begin{equation} \label{2P2f}
	\|u\|_{L^{2p,2}(\Omega)} \le \|v\|_{L^{2p,2}(\Omega^{\star})} \;\;\;\;\;\;\; \text{for all} \;\;\;\;\; 0 < p \le \frac{n}{3n-4}.
	\end{equation}
\end{thm}

\begin{thm} \label{N2}
	Let $n=2$, $f\equiv 1$ in $\Omega$. If $u$ and $v$ are the solutions to problems \eqref{ATPR} and \eqref{SATPR} respectivley. Then
	\begin{equation}
	u^{\star}(x) \le v(x) \,\,\,\,\, x \in \Omega^{\star},
	\end{equation}
where $u^{\star}$ is the convex symmetrization of $u$.
\end{thm}

\begin{thm} \label{F1NM3}
	Let $n\ge 3$ and $f\equiv 1$. If $u$ and $v$ are the solutions to problems \eqref{ATPR} and \eqref{SATPR} respectivley, then 
	\begin{equation} \label{P1}
	\|u\|_{L^{p,1}(\Omega)} \le \|v\|_{L^{p,1}(\Omega^{\star})}
	\end{equation}
	and
	\begin{equation} \label{2P2}
	\|u\|_{L^{2p,2}(\Omega)} \le \|v\|_{L^{2p,2}(\Omega^{\star})},
	\end{equation}
	for all $0 < p \le \frac{n}{n-2}$.
	
\end{thm}

The organization of the paper is the following. In Section 2 the reader can find notations and preliminaries concerning the anisotropic norm, its properties and the convex symmetrization of a function. Section 3 contains results about existence, uniqueness and other properties of the solutions to problem \eqref{ATPR} and \eqref{SATPR}. In section 4 the main theorems are proved. Section 5 is dedicated to the application of these results to prove, in dimension 2, the Bossel-Daners inequality for the anisotropic operator.  Last section will include some counterexamples and problems that still are open.

\section{Notations and preliminaries}
\subsection{Anisotropy}
What follows can be found in \cite{AFLT}. Let $H:\mathbb{R}^n \longrightarrow [0,+\infty]$, $n\ge 2$, be a $C^2(\mathbb{R}^n \setminus \{0\})$ convex functions that satisfies the following homogeneity property
\begin{equation} \label{HP}
H(t\xi)= |t| H(\xi) \,\,\,\,\,\, \forall \xi \in \mathbb{R}^n \, , \,\, \forall t \in \mathbb{R},
\end{equation}
and such that
\begin{equation} \label{control}
\gamma |\xi | \le H(\xi) \le \delta |\xi|,
\end{equation}
for some positive constants $\gamma \le \delta$. Because of \eqref{HP}, we can assume that the set
\begin{equation*}
K = \{\xi \in \mathbb{R}^n : H(\xi)\le 1 \}
\end{equation*}
is such that $|K|$ is equal to the measure $\omega_n$ of the unit sphere in $\mathbb{R}^n$. We can define the support function of $K$ as
\begin{equation}
H^{\circ}(x)= \sup_{\xi \in K} \left< x , \xi \right>,
\end{equation} 
where $\left< \cdot , \cdot \right>$ denotes the scalar product in $\mathbb{R}^n$. $H^{\circ}:\mathbb{R}^n \longrightarrow [0,+\infty]$ is a convex, homogeneous function. Moreover $H$ and $H^{\circ}$ are polar to each other, in the sense that
\begin{equation*}
H^{\circ}(x) = \sup_{\xi \neq 0 } \frac{\left< x , \xi \right>}{H(\xi)}
\end{equation*}
and
\begin{equation*}
H(x) = \sup_{\xi \neq 0 } \frac{\left< x , \xi \right>}{H^{\circ}(\xi)}.
\end{equation*}
$H^{\circ}$ is the support function of the set
\begin{equation*}
K^{\circ} = \{x \in \mathbb{R}^n : H^{\circ}(x)\le 1 \}.
\end{equation*}
The set $\mathcal{W}=\{x \in \mathbb{R}^n : H^{\circ}(x)< 1 \}$ is the so-called Wulff shape centered at the origin. We set $k_n = |\mathcal{W}|$. More generally we will denote by $\mathcal{W}_R(x_0)$ the Wulff shape centered in $x_0 \in \mathbb{R}^n$ with measure $k_n R^n$ the set $R\mathcal{W}+x_0$, and $\mathcal{W}_R(0)=\mathcal{W}_R$.\\
$H$ and $H^{\circ}$ satisfy the following properties:
\begin{equation} \label{prodsc}
H_{\xi}(\xi)\cdot \xi = H(\xi), \;\;\;\; H^{\circ}_{\xi}(\xi)\cdot \xi = H^{\circ}(\xi),
\end{equation} 
\begin{equation}
H(H^{\circ}_{\xi}(\xi))= H^{\circ}(H_{\xi}(\xi))=1 \,\,\,\,\, \forall \xi \in \mathbb{R}^n \setminus \{0\},
\end{equation} 
\begin{equation}
H^{\circ}(\xi)H_{\xi}(H^{\circ}_{\xi}(\xi))=H(\xi)H^{\circ}_{\xi}(H_{\xi}(\xi)) = \xi \,\,\, \forall \xi \in \mathbb{R}^n \setminus \{0\}.
\end{equation}
If $\Omega \subset \mathbb{R}^n$ is an open bounded set with Lipschitz boundary and $E$ is an open subset of $\mathbb{R}^n$, we can give a generalized definition of perimeter of $E$ with respect to the anisotropic norm as follows
\begin{equation*}
	P_H(E,\Omega)= \int_{\partial^* E \cap \Omega} H(\nu ^E)\,d\mathcal{H}^{n-1},
\end{equation*}
where $\partial^*E$ is the reduced boundary of $E$ and $\nu^E$ is its outer normal. Clearly, if $E$ is open, bounded and Lipschitz, then the outer unit normal exists almost everywhere and
\begin{equation}
P_H(E,\mathbb{R}^n):=P_H (E) = \int_{\partial E} H(\nu ^E)\,d\mathcal{H}^{n-1}.
\end{equation}
By \eqref{control} we have that
\begin{equation*}
\gamma P(E) \le P_H(E)\le \delta P(E).
\end{equation*}
In \cite{AB, AFLT} it is shown that if $u\in W^{1,1}(\Omega)$, then for a.e. $t>0$ 
\begin{equation} \label{AFLTAB}
	-\frac{d}{dt}\int_{\{u>t\}} H(\nabla u)\,dx = P_H(\{u>t\},\Omega) = \int_{\partial^*\{u>t\}\cap \Omega} \frac{H(\nabla u)}{|\nabla u|}\,d\mathcal{H}^{n-1}.
\end{equation}
In particular an isoperimetric inequality for the anisotropic perimeter holds (for instance see \cite{AFLT, Bu, DP, FM})
\begin{equation} \label{IIA}
P_H( E ) \ge n k^{\frac{1}{n}}_n |E|^{1-\frac{1}{n}}.
\end{equation}

\subsection{Convex symmetrization}
Let $f: \Omega \longrightarrow [0,+\infty]$ be a misurable function. The decreasing rearrangement $f^*$ of $f$ is defined as follows
\begin{equation*}
f^*(s)= \inf \{t\ge 0 :  |\{x\in \Omega: |f(x)|>t\}|<s\} \,\,\,\,\,\, s\in [0,|\Omega|],
\end{equation*}
which is the generalized inverse function of the distribution function of $f$. We define the convex symmetrization $f^{\star}$ of $f$ as
\begin{equation*}
f^{\star}(x) = f^*(k_n H^{\circ}(x)^n) \,\,\,\,\,\,\,\, x\in \Omega^{\star}.
\end{equation*}
In particular it is well known that the functions $f$, $f^*$ and $f^{\star}$ are equimisurable, i.e.
\begin{equation*}
|\{f>t\}| = |\{f^* > t\}| = |\{f^{\star}>t\}| \,\,\,\,\, t\ge 0.
\end{equation*}
As a consequence, if $f\in L^p(\Omega)$, $p\ge 1$, then $f^* \in L^p([0,|\Omega|])$, $f^{\star}\in L^p(\Omega^{\star})$ and
\begin{equation*}
\| f\|_{L^p(\Omega)}\| = \| f^* \|_{L^p([0,|\Omega|])} = \|f^{\star}\|_{L^p(\Omega^{\star})}.
\end{equation*}
Moreover the Hardy-Littlewood inequality holds (see \cite{K})
\begin{equation}
\int_{\Omega} |f(x)g(x)| \,dx \le \int_0^{|\Omega|} f^*(s)g^*(s)\,ds.
\end{equation}
So, if we consider $g$ as the characteristic function of the set $\{x\in\Omega : u(x)>t\}$, for some misurable function $u:\Omega \rightarrow \mathbb{R}$ and $t\ge 0$, then we get
\begin{equation}
\int_{\{u>t\}} f(x)\,dx \le \int_0^{\mu(t)} f^*(s)\,ds,
\end{equation}
where, again, $\mu(t)$ is the distribution function of $u$.

\section{Existence, uniqueness and properties of the solution}
In this section we want to prove that the solution to problem \eqref{ATPR} exists, it is unique and non-negative.
\subsection{Existence}
Let us consider the following energy functional
\begin{equation} \label{EF}
E[w]= \frac{1}{2}\int_{\Omega} H^2(\nabla w)\,dx + \frac{\beta}{2}\int_{\partial \Omega} H(\nu) w^2\,d\mathcal{H}^{n-1}- \int_{\Omega} fw\,dx, \,\,\,\,\,\, w \in H^1(\Omega).
\end{equation} 
We notice that the Euler-Lagrange equations of this functional is \eqref{ATPR}, hence if $E[\cdot]$ has minima, there will exist a solution to the considered problem.\\
Let us proceed with the classical Calculus of Variation method to prove the existence of a minimum.\\
\textbf{1) Lower bound}. Let us prove that the functional \eqref{EF} is lower bounded. From \eqref{control} and the generalized Young's inequality we have
\begin{equation*}
\begin{split}
	E[u] &\ge \frac{\gamma^2}{2} \|\nabla u\|^2_{L^2(\Omega)} + \frac{\beta \gamma}{2} \|u\|^2_{L^2(\partial \Omega)}-\frac{\epsilon}{2} \|u\|^2_{L^2(\Omega)} - \frac{1}{2\epsilon} \|f\|^2_{L^2(\Omega)}\\
	& \ge C_1 \bigg(\|\nabla u\|^2_{L^2(\Omega)} +  \|u\|^2_{L^2(\partial \Omega)}\bigg) -\frac{\epsilon}{2} \|u\|^2_{L^2(\Omega)} - \frac{1}{2\epsilon} \|f\|^2_{L^2(\Omega)}\\
	&\ge \bigg(C_2-\frac{\epsilon}{2}\bigg) \|u\|^2_{L^2(\Omega)}-\frac{1}{2\epsilon} \|f\|^2_{L^2(\Omega)}.
\end{split}
\end{equation*}
In the last inequality we have used a Poincaré inequality with trace term (see for instance \cite{BGT}). Here $C_2 = C_2(\beta, \gamma, \Omega)$. If we choose $\epsilon$ small enough, then
\begin{equation*}
	E[u] \ge -\frac{1}{2\epsilon} \|f\|^2_{L^2(\Omega)} > -\infty.
\end{equation*}
We have proved in this way that the functional is bounded from below. Let be
\begin{equation}
	m := \inf_{w\in H^1(\Omega)} E[w].
\end{equation}
and let $\{u_k\} \subset H^1(\Omega)$ be a minimizing sequence, i.e.
\begin{equation*}
	\lim_{k\to \infty} E[u_k]= m.
\end{equation*}
We can suppose that $E[u_k]\le m+1$ for all $k\in \mathbb{N}$.\\
\textbf{2) Compactness and lower semicontinuity}. Using again the generalized Young's inequality and the Poincaré inequality with trace term, we have
\begin{equation*}
\begin{split}
	m+1 &\ge \frac{\gamma^2}{4} \|\nabla u_k\|^2_{L^2(\Omega)}-\frac{\epsilon}{2}\|u_k\|^2_{L^2(\Omega)} \\
	&+\frac{\gamma^2}{4} \|\nabla u_k\|^2_{L^2(\Omega)}+ \frac{\beta \gamma}{2}\|u_k\|^2_{L^2(\partial\Omega)} - \frac{1}{2\epsilon} \|f\|^2_{L^2(\Omega)}\\
	&\ge \frac{\gamma^2}{4} \|\nabla u_k\|^2_{L^2(\Omega)}-\frac{\epsilon}{2}\|u_k\|^2_{L^2(\Omega)} +C_3 \|u_k\|^2_{L^2(\Omega)} -  \frac{1}{2\epsilon} \|f\|^2_{L^2(\Omega)},
\end{split}
\end{equation*}
where $C_3=C_3(\beta, \gamma, \Omega)$. Choosing $\epsilon$ small enough and calling $C_4 = \min(\frac{\gamma^2}{4}, C_3-\frac{\epsilon}{2})$ then
\begin{equation*}
	\|u_k\|_{H^1(\Omega)} \le \frac{m+1}{C_4} +\frac{1}{2\epsilon C_4} \|f\|^2_{L^2(\Omega)}<\infty.
\end{equation*}
Hence $\{u_k\}$ is bounded in $H^1(\Omega)$, so it exists a subsequence $\{u_{k_{j}}\} \subset \{u_k\}$ that converges weakly in $H^1(\Omega)$ and strongly in $L^2(\Omega)$ to a function $u \in H^1(\Omega)$. To simplify the notation let us continue to call the subsequence as $\{u_k\}$.\\
By the strict convexity of the functions $t \longrightarrow t^2$ and $H^2$, we have the following inequalities
\begin{equation} \label{convu2}
	u_k^2 \ge u^2 + 2u (u_k-u)
\end{equation}
and
\begin{equation} \label{convH2}
	H^2(\nabla u_k) \ge H^2(\nabla u) + 2 H(\nabla u)H_{\xi}(\nabla u)\cdot (\nabla u_k -\nabla u).
\end{equation}
By \eqref{convu2} and \eqref{convH2}, we have
\begin{equation*}
\begin{split}	
E[u_k] &\ge \frac{1}{2}\int_{\Omega} H^2(\nabla u)\,dx + \frac{\beta}{2} \int_{\partial \Omega} u^2\,d\mathcal{H}^{n-1} - \int_{\Omega} fu_k\,dx \\
&+ \int_{\Omega} H(\nabla u)H_{\xi}(\nabla u)\cdot (\nabla u_k-\nabla u)\,dx + \int_{\partial \Omega} u(u_k-u)\,d\mathcal{H}^{n-1}.
\end{split}
\end{equation*}
By the weak convergence of $u_k$ in $H^1(\Omega)$ and being $H^1(\Omega)$ compactly embedded in $L^2(\partial \Omega)$, passing to the limit for $k\to \infty$ we have that
\begin{equation*}
	E[u] \le m.
\end{equation*}
Eventually $E[u]=m$ and $u$ is a minimum.

\subsection{Uniqueness and non-negativeness} Let us prove now that the minimum of \eqref{EF} is unique. If $u,v \in H^1(\Omega)$, by the strict convexity of the funcion $H^2$ we know that if $t\in [0,1]$, then
\begin{equation} \label{convexityH2}
	H^2(t\nabla u + (1-t)\nabla v) \le t H^2(\nabla u)+(1-t)H^2(\nabla v).
\end{equation}
The equality occurs if and only if $t=0$ or $t=1$. Analogously
\begin{equation} \label{convexityu2}
	[tu+(1-t)v]^2 \le t u^2 + (1-t)v^2, \,\,\,\, t\in [0,1].
\end{equation}
Let $u\in H^1(\Omega)$ be a minimizer of \eqref{EF} and let us suppose that there exists another minimizer $v \in H^1(\Omega)$, such that $u\neq v$. Hence $E[u]=E[v]=m$. Let us denote by $w=u+v$ and choose $t=\frac{1}{2}$, then by \eqref{convexityH2} and \eqref{convexityu2}
\begin{equation*}
	E[w] < \frac{E[u]}{2} + \frac{E[v]}{2} = m.
\end{equation*}
This fact contraddicts the minimality of $u$ and so the minimum must be unique.\\

Let us now show the non-negativeness of the solution. Let $u$ be the unique minimum of \eqref{EF}, namely $E[u]=m$. If we consider $|u|$, by \eqref{HP}, we get
\begin{equation*}
	H^2(\nabla |u|)= H^2\bigg(\frac{u}{|u|} \nabla u\bigg) = H^2(\nabla u).
\end{equation*}
Hence
\begin{equation*}
\begin{split}
	E[|u|]&= \frac{1}{2}\int_{\Omega} H^2(\nabla |u|)\,dx + \frac{\beta}{2}\int_{\partial \Omega} H(\nu)|u|^2\,d\mathcal{H}^{n-1} -\int_{\Omega} f|u|\,dx \\
	& = m + \int_{\Omega} f(u-|u|)\,dx = m+2\int_{\{u\le 0\}}fu \,\,\le m.
\end{split}
	\end{equation*}
By the uniqueness of the minimizer it must be $u=|u|$ in $\Omega$. Eventually $u\ge 0$ in $\Omega$.

\subsection{The anisotropic radial case}
Let us consider problem \eqref{SATPR}. Because of uniqueness of the solution and the symmetry of the problem we look for a solution of the type $v(x)= v (H^{\circ}(x))$. This function solves the following ODE
\begin{equation} \label{RADIAL}
	\begin{cases}
	-\frac{1}{r^{n-1}}(r^{n-1}v'(r))' = f^*(k_n r^n) & r \in (0,R) \\
	v'(0)=0\\
	v'(R)+\beta v(R)=0,
	\end{cases}
\end{equation}
where $r= H^{\circ}(x)$ and $R$ is such that $\Omega^{\star}= \mathcal{W}_R$ and $|\Omega|=|\Omega^{\star}|$.\\ Integrating the first equation in \eqref{RADIAL}, calling $\tilde{t}=k_nt^n$, we get
\begin{equation*}
	v'(r) = -\frac{1}{r^{n-1}}\int_0^r t^{n-1}f^*(k_n t^n)\,dt + C_1= -\frac{1}{nk_nr^{n-1}}\int_0^{k_nr^n} f^*(\tilde{t})\,d\tilde{t}+ C_1.
\end{equation*}
Since $v'(0)=0$ then $C_1=0$. By denoting $\tilde{s}= k_n s^n$, another integration gives
\begin{equation*}
\begin{split}	v(r)= -\int_0^r \frac{1}{nk_ns^{n-1}}\int_0^{k_ns^n} &f^*(\tilde{t})\,d\tilde{t}\,ds +C_2= \\
& -\int_0^{k_nr^n} \frac{1}{n^2k_n^{\frac{2}{n}}\tilde{s}^{2-\frac{2}{n}}}\int_0^{\tilde{s}} f^*(\tilde{t})\,d\tilde{t}\,d\tilde{s} +C_2.
\end{split}
\end{equation*}
From $v'(R)+\beta v(R)=0$ we compute $C_2$, hence
\begin{equation*}
\begin{split}
	v(r)=-\int_0^{k_nr^n}& \frac{1}{n^2k_n^{\frac{2}{n}}\tilde{s}^{2-\frac{2}{n}}}\int_0^{\tilde{s}} f^*(\tilde{t})\,d\tilde{t}\,d\tilde{s}+\\
	&\int_0^{k_nR^n} \frac{1}{n^2k_n^{\frac{2}{n}}\tilde{s}^{2-\frac{2}{n}}}\int_0^{\tilde{s}} f^*(\tilde{t})\,d\tilde{t}\,d\tilde{s}+\frac{1}{\beta n k_nR^{n-1}}\int_0^{k_nR^n} f^*(\tilde{t})\,d\tilde{t}.
	\end{split}
\end{equation*}
Therefore
\begin{equation} \label{SOLRAD}
	v(r)=\int_{k_nr^n}^{k_nR^n} \frac{1}{n^2k_n^{\frac{2}{n}}\tilde{s}^{2-\frac{2}{n}}}\int_0^{\tilde{s}} f^*(\tilde{t})\,d\tilde{t}\,d\tilde{s}+\frac{1}{\beta n k_nR^{n-1}}\int_0^{k_nR^n} f^*(\tilde{t})\,d\tilde{t}.
\end{equation}
In this way we have shown that the unique solution to problem \eqref{SATPR} is radially symmetric with respect to anisotropic norm and its value on the boundary is given by
\begin{equation*}
	v(R)= \frac{1}{\beta n k_nR^{n-1}}\int_0^{k_nR^n} f^*(\tilde{t})\,d\tilde{t} \ge 0.
\end{equation*}

\begin{rem}
We stress that if $f\equiv 1$ in $\Omega$, then by \eqref{SOLRAD} the solution to problem \eqref{SATPR} can be written explicitly as follows
\begin{equation*}
	v(x) = v(H^{\circ}(x))= \frac{R}{\beta n}+ \frac{1}{2n}(R^2-H^{\circ}(x)^2).
\end{equation*}
The solution is a paraboloid with respect to the anisotropic norm. Morover if $H$ is the euclidean norm in $\mathbb{R}^n$, we step back to the classical torsion problem (or Saint Venant problem) with robin boundary conditions, whose radial solution is a concave paraboloid.
\end{rem}
\subsection{level sets and distribution functions} If $u$ is a solution to problem \eqref{ATPR}, we will denote by 
\begin{equation*}
	U_t = \{x \in \Omega : u(x)> t \}
\end{equation*}
for a non-negative real number $t\ge0$. It is clear that if $t \le u_{\min}$, then $U_t = \Omega$ and that if $t > u_{\max}$, then $U_t = \emptyset$. With $u_{\min}$ and $u_{\max}$ we have denoted the minimum and the maximum of $u$ in $\Omega$. We will denote by 
\begin{equation}
	\partial U^{\text{int}}_t =  \Omega \cap \partial U_t, \,\,\,\,\,\,\,\,\,  \,\,\, \partial U^{\text{ext}}_t = \partial \Omega \cap \partial U_t
\end{equation}
the internal and external boundary of $U_t$ with respect to $\Omega$, and by
\begin{equation}
    \mu(t) = |U_t|
\end{equation}
the distribution function of $u$.\\
If $v$ is solution to problem \eqref{SATPR}, for $t\ge 0$, we will indicate with
\begin{equation*}
	V_t = \{x\in \Omega^{\star}: v(x)>t\}, \,\,\,\,\,\, \phi(t) = |V_t|
\end{equation*}
the superlevel sets and the distribution function of $v$, respectively. Furthermore, for $ 0 \le t \le v_{\min}$, $V_t = \Omega$, while for $v_{\min}< t < v_{\max}$, the superlevel sets $V_t$ are Wulff shapes homothetic to $\Omega^{\star}$ and strictly  contained in it. Again, $v_{\min}$ and $v_{\max}$ are the minimum and the maximum of $v$ in $\Omega^{\star}$.

\section{Main results}
Before proving the main results let us state Gronwall Lemma and prove two others lemmata that will have a central importance for what will follow.

\begin{lem} (Gronwall) \label{LEM1}
	Let $\xi (t)$ be a continuously differentiable function satysfing for some non-negative constant C, the following differential inequality
	\begin{equation*}
		\tau \xi'(\tau) \le \xi (\tau) + C, 
	\end{equation*}
for all $\tau \ge \tau_0 > 0$. Then we have

\begin{equation*}
 \xi(\tau) \le \tau \frac{\xi(\tau_0)+C}{\tau_0} + C
\end{equation*}
and
\begin{equation*}
	\xi'(\tau) \le \frac{\xi(\tau_0)+C}{\tau_0},
\end{equation*}
for all $\tau \ge \tau_0$.

\end{lem}
\begin{lem}
	Let $u$ and $v$ be the solutions to problems \eqref{ATPR} and \eqref{SATPR} respectively. Then for a.e. $t>0$ we have
\begin{equation} \label{PHI}
		n^2 k_n^{\frac{2}{n}}\phi (t)^{\frac{2n-2}{n}} = \left( -\phi ' (t) + \frac{1}{\beta}\int_{\partial U^{\text{ext}}_t}\frac{H(\nu)}{u} \,d\mathcal{H}^{n-1}\right) \int_0^{\mu(t)}f^*(s)\,ds
\end{equation}
and
\begin{equation}\label{MU}
		n^2 k_n^{\frac{2}{n}}\mu (t)^{\frac{2n-2}{n}} \le \left( -\mu ' (t) + \frac{1}{\beta}\int_{\partial U^{\text{ext}}_t}\frac{H(\nu)}{u} \,d\mathcal{H}^{n-1}\right) \int_0^{\mu(t)}f^*(s)\,ds.
\end{equation}
\end{lem}
\begin{proof}
	Let $t,h>0$ and let us consider the following test function in $H^1(\Omega)$
	\begin{equation*}
		\varphi_h (x) =
		\begin{cases}
		0 & u \le t \\
		u-t & t < u \le t+h \\
		h & u \ge t+h .
		\end{cases}
	\end{equation*}
Substituting this in \eqref{ELE} we have
\begin{equation*}
	\begin{split}
	&\int_{U_t \setminus U_{t+h}} H(\nabla u) H_{\xi}(\nabla u) \cdot \nabla u \,dx + \beta \int_{\partial U^{\text{ext}}_t\setminus \partial U^{\text{ext}}_{t+h}} H(\nu)(u-t)u\,d\mathcal{H}^{n-1} \\
	&+ \beta h \int_{\partial U^{\text{ext}}_{t+h}} H(\nu)u\,d\mathcal{H}^{n-1} = \int_{U_t \setminus U_{t+h}} f(u-t)dx + h \int_{U_{t+h}} f\,dx.
	\end{split}
\end{equation*}
Applying \eqref{prodsc} in the first integral, dividing by $h$ and applying Coarea Formula, we have for a.e. $t>0$
\begin{equation*}
\begin{split}
&\frac{1}{h}\int_t^{t+h}\int_{\partial U^{\text{int}}_{\tau}} \frac{H^2(\nabla u)}{|\nabla u|} \,d\mathcal{H}^{n-1} \,d\tau + \frac{\beta}{h} \int_{\partial U^{\text{ext}}_t\setminus \partial U^{\text{ext}}_{t+h}} H(\nu)(u-t)u\,d\mathcal{H}^{n-1} \\
&+ \beta \int_{\partial U^{\text{ext}}_{t+h}} H(\nu)u\,d\mathcal{H}^{n-1} = \frac{1}{h}\int_{U_t \setminus U_{t+h}} f(u-t)dx + \int_{U_{t+h}} f\,dx.
\end{split}
\end{equation*}
Passing to the limit for $h\rightarrow 0^+$ we have
\begin{equation*}
	\int_{\partial U^{\text{int}}_t} \frac{H^2(\nabla u)}{|\nabla u|} \,d\mathcal{H}^{n-1} + \beta \int_{\partial U^{\text{ext}}_t}H(\nu) u \,d\mathcal{H}^{n-1} = \int_{U_t}f\,dx.
\end{equation*}
Let us set
\begin{equation}
	g(x)=
	\begin{cases}
	H(\nabla u) \;\;\;\;\;\;\;\; \partial U^{\text{int}}_t\\
	\beta u     \;\;\;\;\;\;\;\;\;\;\;\;\;\; \partial U^{\text{ext}}_t.
	\end{cases}
\end{equation}
We want to end the proof using the anisotropic version of the isoperimetric inequality and to this aim it is necessary to write properly the anisotropic peremeter of $U_t$. Because of the regularity of $\partial \Omega$ we know that $\partial U^{\text{ext}}_t$ is sufficiently regular and a normal vector can be defined. Being $u\in H^1(\Omega)$ and $f\in L^2(\Omega)$, then $\partial U^{\text{int}}_t$ can't have any good regularity property. By \eqref{AFLTAB} we can write for a.e. $t>0$
\begin{equation*}
	P_H(U_t)= \int_{\partial U^{\text{int}}_t} \frac{H(\nabla u)}{|\nabla u|}\,d\mathcal{H}^{n-1} + \int_{\partial U^{\text{ext}}_t} H(\nu)\,d\mathcal{H}^{n-1},
\end{equation*}
where $\nu$ is the outer unit normal to $\Omega$. If we set
\begin{equation}
h(x)=
\begin{cases}
\ds \frac{H(\nabla u)}{|\nabla u|} & \partial U^{\text{int}}_t\\
H(\nu)     & \partial U^{\text{ext}}_t,
\end{cases}
\end{equation}
then
\begin{equation*}
	P_H(U_t) = \int_{\partial U_t} h(x)\,d\mathcal{H}^{n-1}.
\end{equation*}
Furthermore we note that 
\begin{equation} \label{hgf}
	\int_{\partial U_t} h(x)g(x)\,d\mathcal{H}^{n-1}= \int_{U_t}f\,dx.
\end{equation}
Therefore by Schwarz inequality and \eqref{hgf}, we have for a.e. $t>0$
\begin{equation*}
\begin{split}
	P^2_H(U_t) &= \left( \int_{\partial U_t} h(x)\,d\mathcal{H}^{n-1} \right)^2 = \left( \int_{\partial U_t} \sqrt{h(x)g(x)} \sqrt{\frac{h(x)}{g(x)}}\,d\mathcal{H}^{n-1} \right)^2\\
	& \le \int_{\partial U_t} h(x)g(x)\,d\mathcal{H}^{n-1} \left( \int_{\partial U^{\text{int}}_t} \frac{1}{|\nabla u|}\,d\mathcal{H}^{n-1} + \frac{1}{\beta}\int_{\partial U^{\text{ext}}_t}\frac{H(\nu)}{u} \,d\mathcal{H}^{n-1}\right) \\
	&\le  \left( -\mu ' (t) + \frac{1}{\beta}\int_{\partial U^{\text{ext}}_t}\frac{H(\nu)}{u} \,d\mathcal{H}^{n-1}\right) \int_0^{\mu(t)}f^*(s)\,ds.
	\end{split}
\end{equation*}
Hence, by \eqref{IIA}
\begin{equation*}
	n^2 k_n^{\frac{2}{n}}\mu (t)^{\frac{2n-2}{n}} \le \left( -\mu ' (t) + \frac{1}{\beta}\int_{\partial U^{\text{ext}}_t}\frac{H(\nu)}{u} \,d\mathcal{H}^{n-1}\right) \int_0^{\mu(t)}f^*(s)\,ds.
\end{equation*}
If we do the same computations, replacing $v$ with $u$, all the previous inequalities become equalities and we have \eqref{PHI}.\\
In particular, if $f \equiv 1$ in $\Omega$, we have
\begin{equation} \label{MUf1}
n^2 k_n^{\frac{2}{n}} \mu (t)^{\frac{n-2}{n}} \le  -\mu ' (t) + \frac{1}{\beta}\int_{\partial U^{\text{ext}}_t}\frac{H(\nu)}{u} \,d\mathcal{H}^{n-1},
\end{equation}
and
\begin{equation} \label{PHIf1}
n^2 k_n^{\frac{2}{n}} \phi (t)^{\frac{n-2}{n}} =  -\phi ' (t) + \frac{1}{\beta}\int_{\partial U^{\text{ext}}_t}\frac{H(\nu)}{u} \,d\mathcal{H}^{n-1}.
\end{equation}
\end{proof}

\begin{rem} \label{umvm}
	Let us notice that $u_{\min} \le v_{\min}$. Indeed, being the level sets of $v$ homothetic to $\Omega^{\star}$, then, using \eqref{ATPR}, \eqref{SATPR} and the isoperimetric inequality
	\begin{equation} \label{VMMUUM}
	\begin{split}
	v_{\min}P_H(\Omega^{\star}) &= \int_{\partial \Omega^{\star}}  H(\nu)v(x) \,d\mathcal{H}^{n-1}= - \frac{1}{\beta}\int_{\partial \Omega^{\star}} H(\nabla v)H_{\xi}(\nabla v)\cdot \nu \,d\mathcal{H}^{n-1} \\
	&= \frac{|\Omega^{\star}|}{\beta} = \frac{|\Omega|}{\beta} = - \frac{1}{\beta}\int_{\partial \Omega} H(\nabla u)H_{\xi}(\nabla u)\cdot \nu \,d\mathcal{H}^{n-1} \\
	&= \int_{\partial \Omega}  H(\nu)u(x) \,d\mathcal{H}^{n-1} \ge u_{\min} P_H(\Omega) \ge u_{\min} P_H(\Omega^{\star}).
	\end{split}
	\end{equation}
As a consequence for all $0<t <v_{\min}$ we have that
\begin{equation} \label{MUMUFI}
	\mu(t)\le \phi(t)= |\Omega|.
\end{equation}
\end{rem}

\begin{lem} \label{LEM3}
	For all $t\ge v_{\min}$ we have 
	\begin{equation} \label{PHILEM3}
		\int_0^t \tau \left( \int_{\partial V_{\tau}\cap \partial \Omega^{\star}} \frac{H(\nu)}{v(x)} \,d\mathcal{H}^{n-1} \right)\,d\tau = \frac{1}{2\beta} \int_0^{|\Omega|}f^*(s)\,ds
	\end{equation}
	and
	\begin{equation} \label{MULEM3}
	\int_0^t \tau \left( \int_{\partial U^{\text{ext}}_t} \frac{H(\nu)}{u(x)} \,d\mathcal{H}^{n-1} \right)\,d\tau \le  \frac{1}{2\beta} \int_0^{|\Omega|}f^*(s)\,ds
	\end{equation}
\end{lem}
\begin{proof}
	By Fubini's Theorem and using \eqref{ATPR}, we have that
	\begin{equation*}
		\begin{split}
		\int_0^{\infty} \tau \left( \int_{\partial U^{\text{ext}}_t} \frac{H(\nu)}{u(x)} \,d\mathcal{H}^{n-1} \right)\,d\tau &= \int_{\partial \Omega} \left( \int_0^{u(x)} \frac{H(\nu)}{u(x)}\tau \,d\tau \right) \,d\mathcal{H}^{n-1} \\
		&= \int_{\partial \Omega} \frac{H(\nu)u(x)}{2}\,d\mathcal{H}^{n-1} = \frac{1}{2\beta} \int_0^{|\Omega|}f^*(s)\,ds.
		\end{split}
	\end{equation*}
	Analogously,
	\begin{equation*}
		\int_0^{\infty} \tau \left( \int_{\partial V_{\tau}\cap \partial \Omega^{\star}} \frac{H(\nu)}{v(x)} \,d\mathcal{H}^{n-1} \right)\,d\tau = \frac{1}{2\beta} \int_0^{|\Omega|}f^*(s)\,ds.
	\end{equation*}
	
By monotonicity of the integral we have that for $t\ge 0$
	\begin{equation*}
		\int_0^t \tau \left( \int_{\partial U^{\text{ext}}_t} \frac{H(\nu)}{u(x)} \,d\mathcal{H}^{n-1} \right)\,d\tau \le \int_0^{\infty} \tau \left( \int_{\partial U^{\text{ext}}_t} \frac{H(\nu)}{u(x)} \,d\mathcal{H}^{n-1} \right)\,d\tau
	\end{equation*}
and if $t\ge v_{\min}$, then $\partial V_t \cap \partial \Omega^{\star} = \emptyset $. Hence
\begin{equation*}
	\int_0^t \tau \left( \int_{\partial V_{\tau}\cap \partial \Omega^{\star}} \frac{H(\nu)}{v(x)} \,d\mathcal{H}^{n-1} \right)\,d\tau = \int_0^{\infty} \tau \left( \int_{\partial V_{\tau}\cap \partial \Omega^{\star}} \frac{H(\nu)}{v(x)} \,d\mathcal{H}^{n-1} \right)\,d\tau,
\end{equation*}
and we have \eqref{MULEM3}, \eqref{PHILEM3}.\\
If $f\equiv 1$ in $\Omega$, then
\begin{equation} \label{MULEM3f1}
\int_0^t \tau \left( \int_{\partial U^{\text{ext}}_t} \frac{H(\nu)}{u(x)} \,d\mathcal{H}^{n-1} \right)\,d\tau \le  \frac{|\Omega|}{2\beta}
\end{equation}
and
\begin{equation} \label{PHILEM3f1}
\int_0^t \tau \left( \int_{\partial V_{\tau}\cap \partial \Omega^{\star}} \frac{H(\nu)}{v(x)} \,d\mathcal{H}^{n-1} \right)\,d\tau = \frac{|\Omega|}{2\beta}.
\end{equation}

\end{proof}

\begin{proof}[Proof of Theorem \ref{FNM2}]
	Let $0< p \le \frac{n}{2n-2}$ and let us denote $K_n = n^2 k_n^{\frac{2}{n}}$. Let us multiply \eqref{MU} by $t\mu(t)^{\eta}$, where $\eta = \frac{1}{p}- \frac{2n-2}{n} \ge 0$, and integrate from $0$ to $\tau \ge v_{\min}$
	\begin{equation*}
	\begin{split}
	\int_0^{\tau} K_n t \mu(t)^{\frac{1}{p}}\,dt &\le \int_0^{\tau} -\mu '(t) t \mu(t)^{\eta}\bigg(\int_0^{\mu(t)}f^*(s)\,ds\bigg) \,dt \,\,\, + \\
	 & \frac{1}{\beta} \int_0^{\tau}t\mu(t)^{\eta}\bigg(\int_{\partial U_t^{\text{ext}}} \frac{H(\nu)}{u(x)} \,d \mathcal{H}^{n-1} \int_0^{\mu(t)}f^*(s)\,ds\bigg)\,dt \\
	&\le \int_0^{\tau} - t \mu(t)^{\eta}\bigg(\int_0^{\mu(t)}f^*(s)\,ds\bigg) \,d\mu(t) + \frac{|\Omega|^{\eta}}{2 \beta^2} \bigg( \int_0^{|\Omega|}f^*(s)\,ds \bigg)^2 , \\
	\end{split}
	\end{equation*}
	where we applied Lemma \ref{LEM3} and the fact that $\mu(t)$ is a monotone non increasing function.\\
	By setting $ \ds F(l)= \int_0^l w^\eta \int_0^w f^*(s)\,ds \,dw$ and integrating by parts the first and last members in this chain of inequalities we have
	\begin{equation*}
	\begin{split}
	\tau F(\mu(\tau))+ \tau \int_0^{\tau} K_n  \mu(t)^{\frac{1}{p}}\,dt \le \int_0^{\tau} &F(\mu(t))\,dt + \int_0^{\tau}\int_0^{t} K_n  \mu(t)^{\frac{1}{p}}\,dr \,dt  \\
	&+ \frac{|\Omega|^{\eta}}{2 \beta^2} \bigg( \int_0^{|\Omega|}f^*(s)\,ds \bigg)^2
	\end{split}
	\end{equation*}
	By applying Lemma \ref{LEM1}, with
	\begin{equation*}
	\xi (\tau) = \int_0^{\tau} F(\mu(t))\,dt + \int_0^{\tau}\int_0^{t} K_n \mu(t)^{\frac{1}{p}}\,dr \,dt,
	\end{equation*}
	$ \ds C= \frac{|\Omega|^{\eta}}{2\beta^2} \bigg( \int_0^{|\Omega|}f^*(s)\,ds \bigg)^2$ and $\tau_0 = v_{\min}$, we have that
	\begin{equation} \label{FMU}
	\begin{split}
	F(\mu(\tau))&+ \int_0^{\tau} K_n  \mu(t)^{\frac{1}{p}}\,dt \le \frac{1}{v_{\min}} \bigg( \int_0^{v_{\min}} F(\mu(t))\,dt \\
	&+ \int_0^{v_{\min}}\int_0^t K_n  \mu(r)^{\frac{1}{p}}\,dr \,dt + \frac{|\Omega|^{\eta}}{2\beta^2} \bigg( \int_0^{|\Omega|}f^*(s)\,ds \bigg)^2 \bigg).
	\end{split}
	\end{equation}
	Analogously
	\begin{equation} \label{FPHI}
	\begin{split}
	F(\phi(\tau))&+ \int_0^{\tau} K_n  \phi(t)^{\frac{1}{p}}\,dt = \frac{1}{v_{\min}} \bigg( \int_0^{v_{\min}} F(\phi(t))\,dt \\
	&+ \int_0^{v_{\min}}\int_0^t K_n  \phi(r)^{\frac{1}{p}}\,dr \,dt + \frac{|\Omega|^{\eta}}{2\beta^2}\bigg( \int_0^{|\Omega|}f^*(s)\,ds \bigg)^2 \bigg).
	\end{split}
	\end{equation}
	By \eqref{VMMUUM} and \eqref{MUMUFI}, then we can compare directly the righthand sides of \eqref{FMU} and \eqref{FPHI}. So
	\begin{equation*}
	F(\mu(\tau))+ \int_0^{\tau} K_n  \mu(t)^{\frac{1}{p}}\,dt \le F(\phi(\tau))+ \int_0^{\tau} K_n  \phi(t)^{\frac{1}{p}}\,dt
	\end{equation*}
	For $\tau \rightarrow +\infty$ we have
	\begin{equation*}
	\int_0^{\infty}  \mu(t)^{\frac{1}{p}}\,dt \le \int_0^{\infty}  \phi(t)^{\frac{1}{p}}\,dt,
	\end{equation*}
	which is \eqref{P1f}.\\
	Now we want to prove \eqref{2P2f}. In order to obtain this result let us pass to the limit as $\tau \rightarrow \infty$ the following inequality:
	\begin{equation*}
	\int_0^{\tau} K_n t \mu(t)^{\frac{1}{p}}\,dt \le \int_0^{\tau} -t \mu(t)^{\eta} \bigg( \int_0^{|\Omega|}f^*(s)\,ds \bigg)\,d\mu(t) +\frac{|\Omega|^{\eta}}{2\beta^2}\bigg( \int_0^{|\Omega|}f^*(s)\,ds \bigg)^2 .
	\end{equation*}
	After an integration by parts we get
	\begin{equation*}
	\int_0^{\infty} K_n t \mu(t)^{\frac{1}{p}}\,dt \le \int_0^{\infty} F(\mu(t))\,dt +\frac{|\Omega|^{\eta}}{2\beta^2}\bigg( \int_0^{|\Omega|}f^*(s)\,ds \bigg)^2.
	\end{equation*}
	On the other hand
	\begin{equation*}
	\int_0^{\infty} K_n t \phi(t)^{\frac{1}{p}}\,dt = \int_0^{\infty} F(\phi(t))\,dt +\frac{|\Omega|^{\eta}}{2\beta^2}\bigg( \int_0^{|\Omega|}f^*(s)\,ds \bigg)^2.
	\end{equation*}
	So we need just to show that
	\begin{equation} \label{GMUGPHI}
	\int_0^{\infty} F(\mu(t))\,dt \le \int_0^{\infty} F(\phi(t))\,dt.
	\end{equation}
	To this aim we multiply \eqref{MU} by $tF(\mu(t))\mu(t)^{-\frac{2n-2}{2}}$. Since $F(l)l^{-\frac{2n-2}{2}}$ is a non decreasing function in $l$, when $ 0 < p \le \frac{n}{3n-4}$, we can integrate from $0$ to $\tau \ge v_{\min}$, to obtain
	\begin{equation*}
	\begin{split}
	\int_0^{\tau} &K_n t F(\mu(t))\,dt \le \int_0^{\tau}-t\mu(t)^{-\frac{2n-2}{n}}F(\mu(t)) \bigg( \int_0^{\mu(t)}f^*(s)\,ds \bigg)\,d\mu(t) \\
	&+ \frac{1}{\beta} \int_0^{\tau} tF(\mu(t))\mu(t)^{-\frac{2n-2}{n}}\int_{\partial U_t^{\text{ext}}} \frac{H(\nu)}{u}\,d\mathcal{H}^{n-1} \bigg( \int_0^{\mu(t)}f^*(s)\,ds \bigg)\,dt \\
	&\le \int_0^{\tau}-t\mu(t)^{-\frac{2n-2}{n}}F(\mu(t)) \bigg( \int_0^{\mu(t)}f^*(s)\,ds \bigg)\,d\mu(t)\\
	& + \frac{1}{\beta} F(|\Omega|)|\Omega|^{-\frac{2n-2}{n}}  \bigg( \int_0^{|\Omega|}f^*(s)\,ds \bigg)\int_0^{\tau} t\int_{\partial U_t^{\text{ext}}} \frac{H(\nu)}{u}\,d\mathcal{H}^{n-1} \,dt\\
	&\le \int_0^{\tau}-t\mu(t)^{-\frac{2n-2}{n}}F(\mu(t))\bigg( \int_0^{\mu(t)}f^*(s)\,ds \bigg)\,d\mu(t)\\
	& + F(|\Omega|)\frac{|\Omega|^{\frac{2n-2}{n}}}{2\beta^2}\bigg( \int_0^{|\Omega|}f^*(s)\,ds \bigg)^2,
	\end{split}
	\end{equation*}
	where, again, we applied \ref{LEM3}. Now, if we call $\ds C=F(|\Omega|)\frac{|\Omega|^{\frac{2n-2}{n}}}{2\beta^2}\bigg( \int_0^{|\Omega|}f^*(s)\,ds \bigg)^2$ and set $\ds J(l)= \int_0^l w^{-\frac{2n-2}{n}}F(w)\bigg( \int_0^{w}f^*(s)\,ds \bigg)\,dw$, integrating by parts the first and last member of the previous chain of inequalities, we have
	\begin{equation*}
	\tau \int_0^{\tau}K_n F(\mu(t))\,dt + \tau J(\mu(\tau)) \le \int_0^{\tau} \int_0^r K_n F(\mu(z))\,dz \,dr + \int_0^{\tau} J(\mu(t))\,dt + C.
	\end{equation*}
	Setting
	\begin{equation*}
	\xi (\tau) = \int_0^{\tau} \int_0^r K_n F(\mu(z))\,dz \,dr + \int_0^{\tau} J(\mu(t))\,dt,
	\end{equation*}
	and applying \ref{LEM1} with $\tau_0 = v_{\min}$ we deduce that
	\begin{equation*}
	\begin{split}
	\int_0^{\tau} K_n F(\mu(t))\,dt + J(\mu(\tau)) \le \frac{1}{v_{\min}} \bigg( \int_0^{v_{\min}}&\int_0^r K_n F(\mu(z))\,dz \,dr \\
	&+ \int_0^{v_{\min}} J(\mu(t))\,dt + C \bigg).
	\end{split}
	\end{equation*}
	This inequality holds as an equality when we have $\phi$ in place of $\mu$, so as before
	\begin{equation*}
	\int_0^{\tau} K_n F(\mu(t))\,dt + J(\mu(t)) \le \int_0^{\tau} K_n F(\phi(t))\,dt + J(\phi(t)).
	\end{equation*}
	For $\tau \rightarrow \infty$ we have \eqref{2P2f}, which concludes the proof.
\end{proof}

\begin{proof}[Proof of Theorem \ref{N2}]
	Multiplying by $t\ge 0$  inequality \eqref{MUf1} and integrating from $0$ to $\tau \ge v_{\min}$, we have that
	\begin{equation*}
		2k_2 \tau^2 \le \int_0^{\tau} -\mu '(t)t \,dt + \frac{|\Omega|}{2\beta^2}.
	\end{equation*}
Here we applied Lemma \ref{LEM3}. Analogously for \eqref{PHIf1}
\begin{equation*}
	2k_2 \tau^2 = \int_0^{\tau} -\phi '(t)t \,dt + \frac{|\Omega|}{2\beta^2}.
\end{equation*}
Then
\begin{equation*}
	\int_0^{\tau} t(-d\mu(t)) \ge \int_0^{\tau} t(-d\phi(t)),
\end{equation*}
for every $\tau\ge v_{\min}$. Integrating by parts
\begin{equation} \label{INTAU}
	\mu(\tau) \le \phi(\tau) \,\,\,\,\,\, \tau\ge v_{\min}.
\end{equation}
Since $u_{\min} \le v_{\min}$, inequality \eqref{INTAU} holds for $t\ge 0$ and the claim is proved.
\end{proof}

	

\begin{proof}[Proof of Theorem \ref{F1NM3}]
Let $0<p \le \frac{n}{n-2}$. Let us multiply \eqref{MUf1} by $t\mu(t)^{\eta}$, where $\eta = \frac{1}{p}- \frac{n-2}{n} \ge 0$, and integrate from $0$ to $\tau \ge v_{\min}$
\begin{equation*}
\begin{split}
	\int_0^{\tau} K_n t \mu(t)^{\frac{1}{p}}\,dt &\le \int_0^{\tau} -\mu '(t) t \mu(t)^{\eta}\,dt + \frac{1}{\beta} \int_0^{\tau}t\mu(t)^{\eta}\int_{\partial U_t^{\text{ext}}} \frac{H(\nu)}{u(x)} \,d \mathcal{H}^{n-1}\\
	&\le \int_0^{\tau} -\mu '(t) t \mu(t)^{\eta}\,dt + \frac{|\Omega|^{\eta}}{\beta} \int_0^{\tau}t\int_{\partial U_t^{\text{ext}}} \frac{H(\nu)}{u(x)} \,d \mathcal{H}^{n-1} \\
	&\le \int_0^{\tau} -\mu '(t) t \mu(t)^{\eta}\,dt +\frac{|\Omega|^{\eta +1}}{2\beta^2} \le \int_0^{\tau} -t \mu(t)^{\eta}\,d\mu(t) +\frac{|\Omega|^{\eta +1}}{2\beta^2}
\end{split}
\end{equation*}
where, again, $K_n = n^2 k_n^{\frac{2}{n}}$, in the third inequality we applied Lemma \ref{LEM3}, and in the last the fact that $\mu(t)$ is a monotone non increasing function.\\
By setting $ \ds G(l)= \int_0^l w^\eta = \frac{l^{\eta + 1}}{\eta + 1}$ and integrating by parts the first and last members in this chain of inequalities we have
\begin{equation*}
	\tau G(\mu(\tau))+ \tau \int_0^{\tau} K_n t \mu(t)^{\frac{1}{p}}\,dt \le \int_0^{\tau} G(\mu(t))\,dt + \int_0^{\tau}\int_0^{t} K_n  \mu(t)^{\frac{1}{p}}\,dr \,dt + \frac{|\Omega|^{\eta +1}}{2\beta^2}
\end{equation*}
By applying Lemma \ref{LEM1}, with
\begin{equation*}
	\xi (t) = \int_0^{\tau} G(\mu(t))\,dt + \int_0^{\tau}\int_0^{t} K_n \mu(t)^{\frac{1}{p}}\,dr \,dt,
\end{equation*}
$\ds C= \frac{|\Omega|^{\eta +1}}{2\beta^2}$ and $\tau_0 = v_{\min}$, we have that
\begin{equation} \label{GMU}
	\begin{split}
	 G(\mu(\tau))&+ \int_0^{\tau} K_n  \mu(t)^{\frac{1}{p}}\,dt \le \frac{1}{v_{\min}} \bigg( \int_0^{v_{\min}} G(\mu(t))\,dt \\
	 &+ \int_0^{v_{\min}}\int_0^t K_n  \mu(r)^{\frac{1}{p}}\,dr \,dt + \frac{|\Omega|^{\eta +1}}{2\beta^2} \bigg).
	\end{split}
\end{equation}
Analogously
\begin{equation} \label{GPHI}
\begin{split}
G(\phi(\tau))&+ \int_0^{\tau} K_n  \phi(t)^{\frac{1}{p}}\,dt = \frac{1}{v_{\min}} \bigg( \int_0^{v_{\min}} G(\phi(t))\,dt \\
&+ \int_0^{v_{\min}}\int_0^t K_n  \phi(r)^{\frac{1}{p}}\,dr \,dt + \frac{|\Omega|^{\eta +1}}{2\beta^2} \bigg).
\end{split}
\end{equation}
By \eqref{VMMUUM} and \eqref{MUMUFI}, we compare directly the righthand sides of \eqref{GMU} and \eqref{GPHI}. So
\begin{equation*}
	G(\mu(\tau))+ \int_0^{\tau} K_n  \mu(t)^{\frac{1}{p}}\,dt \le G(\phi(\tau))+ \int_0^{\tau} K_n  \phi(t)^{\frac{1}{p}}\,dt
\end{equation*}
For $\tau \rightarrow +\infty$ we have
\begin{equation*}
	\int_0^{\infty}  \mu(t)^{\frac{1}{p}}\,dt \le \int_0^{\infty}  \phi(t)^{\frac{1}{p}}\,dt,
\end{equation*}
which is \eqref{P1}.
Now we want to prove \eqref{2P2}. In order to obtain this result let us pass to the limit as $\tau \rightarrow \infty$ the following inequality:
\begin{equation*}
	\int_0^{\tau} K_n t \mu(t)^{\frac{1}{p}}\,dt \le \int_0^{\tau} -t \mu(t)^{\eta}\,d\mu(t) +\frac{|\Omega|^{\eta +1}}{2\beta^2}.
\end{equation*}
After an integration by parts we get
\begin{equation*}
	\int_0^{\infty} K_n t \mu(t)^{\frac{1}{p}}\,dt \le \int_0^{\infty} G(\mu(t))\,dt +\frac{|\Omega|^{\eta +1}}{2\beta^2}.
\end{equation*}
On the other hand
\begin{equation*}
	\int_0^{\infty} K_n t \phi(t)^{\frac{1}{p}}\,dt = \int_0^{\infty} G(\phi(t))\,dt +\frac{|\Omega|^{\eta +1}}{2\beta^2}.
\end{equation*}
So we need just to show that
\begin{equation} \label{GMUGPHI}
	\int_0^{\infty} G(\mu(t))\,dt \le \int_0^{\infty} G(\phi(t))\,dt.
\end{equation}
To this aim we multiply \eqref{MUf1} by $tG(\mu(t))\mu(t)^{-\frac{n-2}{n}}$. Since $G(l)l^{-\frac{n-2}{n}}$ is a non decreasing function in $l$, we can integrate from $0$ to $\tau \ge v_{\min}$, to obtain
\begin{equation*}
\begin{split}
\int_0^{\tau} &K_n t G(\mu(t))\,dt \le \int_0^{\tau}-t\mu(t)^{-\frac{n-2}{n}}G(\mu(t))\,d\mu(t) \\
&+ \frac{1}{\beta} \int_0^{\tau} tG(\mu(t))\mu(t)^{-\frac{n-2}{n}}\int_{\partial U_t^{\text{ext}}} \frac{H(\nu)}{u}\,d\mathcal{H}^{n-1}\,dt \\
&\le \int_0^{\tau}-t\mu(t)^{-\frac{n-2}{n}}G(\mu(t))\,d\mu(t) + \frac{1}{\beta} G(|\Omega|)|\Omega|^{-\frac{n-2}{n}} \int_0^{\tau} t\int_{\partial U_t^{\text{ext}}} \frac{H(\nu)}{u}\,d\mathcal{H}^{n-1} \,dt\\
&\le \int_0^{\tau}-t\mu(t)^{-\frac{n-2}{n}}G(\mu(t))\,d\mu(t) + G(|\Omega|)\frac{|\Omega|^{\frac{2}{n}}}{2\beta^2},
\end{split}
\end{equation*}
where, again, we applied \ref{LEM3}. Now, if we call $\ds C=G(|\Omega|)\frac{|\Omega|^{\frac{2}{n}}}{2\beta^2}$ and set $\ds J(l)= \int_0^l w^{-\frac{n-2}{n}}G(w)\,dw$, integrating by parts the first and last member of the previous chain of inequalities, we have
\begin{equation*}
	\tau \int_0^{\tau}K_n G(\mu(t))\,dt + \tau J(\mu(\tau)) \le \int_0^{\tau} \int_0^r K_n G(\mu(z))\,dz \,dr + \int_0^{\tau} J(\mu(t))\,dt + C.
\end{equation*}
Setting
\begin{equation*}
	\xi (t) = \int_0^{\tau} \int_0^r K_n G(\mu(z))\,dz \,dr + \int_0^{\tau} J(\mu(t))\,dt,
\end{equation*}
and applying \ref{LEM1} with $\tau_0 = v_{\min}$ we deduce that
\begin{equation*}
\begin{split}
\int_0^{\tau} K_n G(\mu(t))\,dt + J(\mu(\tau)) \le \frac{1}{v_{\min}} \bigg( \int_0^{v_{\min}}&\int_0^r K_n G(\mu(z))\,dz \,dr \\
&+ \int_0^{v_{\min}} J(\mu(t))\,dt, + C \bigg).
\end{split}
\end{equation*}
This inequality holds as an equality when we have $\phi$ in place of $\mu$, so as before
\begin{equation*}
	\int_0^{\tau} K_n G(\mu(t))\,dt + J(\mu(t)) \le \int_0^{\tau} K_n G(\phi(t))\,dt + J(\phi(t)).
\end{equation*}
For $\tau \rightarrow \infty$ we have \eqref{GMUGPHI}, which concludes the proof.
\end{proof}

\section{Application to PDE's: Bossel-Daners inequality} 
Let $\Omega$ be a bounded and smooth open set in $\mathbb{R}^n$. Let us denote by $\nu$ the outer unit normal to $\partial \Omega$ and let $\beta>0$ be a positive real number. It is well known that for the following Laplacian eigenvalue problem with Robin boundary conditions
\begin{equation*}
	\begin{cases}
	-\Delta u = \lambda(\Omega)u & \text{in} \,\, \Omega \\
	\ds \frac{\partial u}{\partial \nu} + \beta u = 0 & \text{on}\,\, \partial \Omega,
	\end{cases}
\end{equation*}
a Faber-Krahn type inequality for the first eigenvalue holds. It is famous under the name of Bossel-Daners inequality and it can be read as follows
\begin{equation*}
	\lambda_{1,\beta}(\Omega) \ge \lambda_{1,\beta} (\Omega^{\sharp}),
\end{equation*}
where $\Omega^{\sharp}$ is the ball centered at the origin with the same measure as $\Omega$. Equality holds if and only if $\Omega$ is a ball.\\
Let us consider now the anisotropic case. If $f=\lambda(\Omega) u$, then \eqref{ATPR} can be written in this way
\begin{equation} \label{ATPREP}
	\begin{cases}
	-\mathrm{div} (H(\nabla u)H_{\xi}(\nabla u))=\lambda(\Omega) u & \mbox{in}\ \Omega\vspace{0.2cm}\\
	H(\nabla u)H_{\xi}(\nabla u)\cdot \nu +\beta H(\nu) u=0&\mbox{on}\ \partial \Omega\vspace{0.2cm}.
	\end{cases}
\end{equation}
The variational characterization for the first eiganvalue is
\begin{equation}
	\lambda_1(\Omega) = \min_{w \in H^1(\Omega)\setminus\{0\}} J[u],
\end{equation}
where
\begin{equation}
	\ds J[u]= \frac{\ds \int_{\Omega} H^2(\nabla u)\,dx + \beta \int_{\partial \Omega}u^2 H(\nu) \,d\mathcal{H}^{n-1}}{\ds \int_{\Omega} u^2\,dx}.
\end{equation}
In \cite{DG} the authors proved a Bossel-Daners type inequality for the anisotropic $p$-Laplacian problem. Indeed, they proved that
\begin{equation} \label{ABDI}
	\lambda_{1,\beta} (\Omega) \ge \lambda_{1,\beta} (\Omega^{\star}),
\end{equation}
where $\Omega^{\star}$ is the set homothetic to the Wulff Shape having the same measure as $\Omega$. In particular the equality case holds if and only if $\Omega$ is a set homothetic to the Wulff Shape.\\
In this section we want to give an alternative proof of \eqref{ABDI} in the planar case, using the results found in the previous section.
\begin{cor}
Let $n=2$ and $u$ be the first eigenfunction associated to $\lambda_{1,\beta}(\Omega)$, solution to problem \eqref{ATPREP}. If $z$ is the solution to the symmetrized problem
\begin{equation}
	\begin{cases}
	-\mathrm{div} (H(\nabla z)H_{\xi}(\nabla z))=\lambda_{1,\beta}(\Omega) u^{\star} & \mbox{in}\ \Omega^{\star}\vspace{0.2cm}\\
	H(\nabla z)H_{\xi}(\nabla z)\cdot \nu +\beta H(\nu) z=0&\mbox{on}\ \partial \Omega^{\star}\vspace{0.2cm}.
	\end{cases}
\end{equation}
then we have that
\begin{equation}
	\lambda_{1,\beta} (\Omega) \ge \lambda_{1,\beta} (\Omega^{\star}).
\end{equation}
Equality occurs if and only if $\Omega$ is homothetic to the Wulff Shape with same measure as $\Omega$.
\end{cor}
\begin{proof} By theorem \ref{N2}, we know that
	\begin{equation*}
		\int_{\Omega} u^2\,dx = \int_{\Omega^{\star}} (u^{\star})^2 \,dx \le \int_{\Omega^{\star}} z^2\,dx.
	\end{equation*}
So, by Cauchy-Schwarz inequality we have
\begin{equation*}
	\int_{\Omega^{\star}} u^{\star}z\,dx \le \int_{\Omega^{\star}} z^2\,dx,
\end{equation*}
and this ends the corollary, indeed
\begin{equation*}
\begin{split}
	\lambda_{1,\beta}(\Omega) &= \frac{\ds \int_{\Omega} H^2(\nabla z)\,dx + \beta \int_{\partial \Omega}z^2 H(\nu) \,d\mathcal{H}^{n-1}}{\ds \int_{\Omega} u^{\star}z\,dx}\\
	& \ge \frac{\ds \int_{\Omega} H^2(\nabla z)\,dx + \beta \int_{\partial \Omega}z^2 H(\nu) \,d\mathcal{H}^{n-1}}{\ds \int_{\Omega} z^2\,dx} = \lambda_{1,\beta}(\Omega^{\star}).
	\end{split}
\end{equation*}
\end{proof}

\section{Conclusions and open problems}
As in the euclidean case, we have proved that these comparison results depend on the dimension of the space. In particular if we are in the hypohtesis of theorem \ref{FNM2}, when $n=2$, then
\begin{equation*}
	\|u\|_{L^1 (\Omega)} \le \|v\|_{L^1(\Omega^{\star})},
\end{equation*}
and
\begin{equation*}
\|u\|_{L^2 (\Omega)} \le \|v\|_{L^2(\Omega^{\star})}.
\end{equation*}
Therefore a question arises spontaneously. Is it true that
\begin{equation} \label{pestimates}
	\|u\|_{L^p (\Omega)} \le \|v\|_{L^p(\Omega^{\star})}
\end{equation}
for all values of $p$? In dimension $2$ the answer is negative for large values of $p$. Next example will show that \eqref{pestimates} is untrue when $p=\infty$ and $n=2$.

\begin{exa} \label{ex1}
Let $\Omega$ be the union of two disjoint bidimensional Wulff shapes $\mathcal{W}$ and $\mathcal{W}_r$, with radii $1$ and $r$ respectively. If we choose $\beta = \frac{1}{2}$ and $f$ such that it is constantly $1$ in $\mathcal{W}$ and constantly zero in $\mathcal{W}_r$, then the solutions to problem \eqref{ATPR} and \eqref{SATPR} can be explicitly computed. In particular it is possible to prove that there exitsts a positive constant $c$ such that
\begin{equation*}
	\|u\|_{L^{\infty}(\Omega)} - \|v\|_{L^{\infty}(\Omega^{\star})} = cr^2 + o(r^2).
\end{equation*}
\end{exa}

Now someone could ask if \eqref{pestimates} can be true when $n\ge 3$. Next counterexample will show its untruthfulness  when $n=3$ and $p=2$.
\begin{exa}
If we consider $\Omega$, $\beta$ and $f$ as in the example \ref{ex1} in the corresponding three-dimensional case, then, in the hypothesis of theorem \ref{FNM2}, the solutions to problem \eqref{ATPR} and \eqref{SATPR} can be explicitly computed. It is possible to prove that there exists a positive constant $d$ such that
\begin{equation*}
\|u\|_{L^2(\Omega)} - \|v\|_{L^2(\Omega^{\star})} = dr^3 + o(r^3).
\end{equation*}
\end{exa}
A problem that is still open is the following 
\begin{oppr}
	In the hypothesis of theorem \ref{FNM2}, \eqref{pestimates} is true for $p=1$ and $n\ge 3$?
\end{oppr}

If we now consider the  theorem \ref{N2}, we have proved that when $n=2$ and $f\equiv 1$ in $\Omega$, then
\begin{equation*} \label{starestimate}
	u^{\star}(x)\le v(x) \,\,\,\, x \in \Omega^{\star}.
\end{equation*}
In doing so, another question arises:
\begin{oppr}
In the hypothesis of theorem \ref{N2}, \eqref{starestimate} is true even when $n\ge 3$?
\end{oppr}
\small{

\end{document}
\begin{thebibliography}{999}
		
		
		\bibitem[AB]{AB} M. Amar, G. Bellettini, \textit{A notion of total variation depending on a metric with
		discontinuous coefficients}, Ann. Inst. Henri Poincaré, Anal. Nonlinéaire 11 (1994), no.1, 91-133. 
		
		\bibitem[AFLT]{AFLT} A. Alvino, V. Ferone, P. L. Lions, G. Trombetti, \textit{Convex symmetrization and applicatons},
		Annales de l'I.H.P. 14 (1997), 275-293-65. 
		
		\bibitem[ALT]{ALT}  A. Alvino, P. L. Lions, G. Trombetti, \textit{Comparison results for elliptic and parabolic equations via Schwarz symmetrization}, Annales de l'I.H.P. 7 (1990), 37-65.
		
		\bibitem[ANT]{ANT} A. Alvino, C. Nitsch, C. Trombetti, \textit{A Talenti comparison result for solutions to elliptic problems with robin boundary conditions}, ArXiv (2019): 1-15.
		
		
		\bibitem[Bu]{Bu} H. Busemann, \textit{The isoperimetric problem for Minkowski area}, Amer. J. Math. 71 (1949)
		743-762. 
		
		\bibitem[BG]{BG} D. Bucur, A. Giacomini, \textit{ Faber-Krahn inequalities for the Robin-Laplacian: A free discontinuity approach}, Arch. Rational Mech. Anal. 218, no. 2, (2015), 757-824.
		
		\bibitem[BGT]{BGT} D. Bucur, A. Giacomini, P. Trebeschi \textit{Best constant in Poincaré inequalities with traces: a free discontinuity approach}, Ann. Inst. H. Poincaré Anal. Non Linéaire 36 (2019), no. 7, 1959-1986.
		
		
		\bibitem[DG]{DG}  F. Della Pietra, N. Gavitone, \textit{Faber-Krahn inequality for anisotropic eigenvalue problems with Robin boundary
		conditions}, Potential Analysis, 41, n. 4 (2014), 1147-1166.
		
		\bibitem[DP]{DP}  B. Dacorogna, C.E. Pfister, \textit{Wulff theorem and best constant in Sobolev inequality}, J. Math.
		Pures Appl. (9) 71 (1992) no. 2, 97-118.
		
		\bibitem[FM]{FM}  I. Fonseca, S. Müller, \textit{A uniqueness proof for the Wulff theorem}, Proc. Roy. Soc. Edinburgh
		Sect. A 119 (1991) no. 1-2, 125-136.
		
		
		
		\bibitem[K]{K} S. Kesavan, \textit{Symmetrization \& applications}, Series in Analysis, 3, World Scientific Publishing
		Co. Pte. Ltd., (2006).
	
		
		
		\bibitem[L]{L} G. G. Lorentz, \textit{Some new functional spaces}, Ann. of Math. (2) 51, (1950), 37-55.
		
		
		
		
		no. 27, Princeton University Press, (1951).
		
		\bibitem[S]{S} R. Sannipoli, \textit{Some properties of the torsion function with Robin boundary conditions}, ArXiv (2021): 1-18.
		
		\bibitem[T]{T} G. Talenti, \textit{Elliptic equations and rearrangements}, Ann. Scuola Norm. Sup. Pisa Cl. Sci. (4)
		3, no. 4, (1976), 697-718.
		
		\bibitem[T1]{T1} G. Talenti, \textit{Nonlinear elliptic equations, rearrangements of functions and Orlicz spaces}, Ann. Mat. Pura e Appl. 120 (1979), 159-184.
		
		\bibitem[T2]{T2}  G. Talenti. \textit{On the first eigenvalue of the clamped plate}, Ann. Mat. Pura e Appl. 129 (1981), 265-280.
		
		
		
		
		
		
	\end{thebibliography}
